\documentclass{amsart}

\usepackage{enumerate}
\usepackage{amsmath}

\usepackage{graphicx}
\usepackage{caption}
\usepackage{subcaption}

\newcommand\blfootnote[1]{%
	\begingroup
	\renewcommand\thefootnote{}\footnote{#1}%
	\addtocounter{footnote}{-1}%
	\endgroup
}

\newtheorem{theorem}{Theorem}[section]
\newtheorem{lemma}[theorem]{Lemma}

\newtheorem{corollary}[theorem]{Corollary}

\theoremstyle{definition}
\newtheorem{definition}[theorem]{Definition}
\newtheorem{example}[theorem]{Example}

\theoremstyle{remark}

\theoremstyle{Property}

\numberwithin{equation}{section}

\begin{document}

\title[Counting Hyperbolic Components in the Main Molecule]{Counting Hyperbolic Components \\ in the Main Molecule}

\author{Schinella D'Souza}
\address{Department of Mathematics, University of Michigan, Ann Arbor,  Michigan 48109-1043}
\email{dsouzas@umich.edu}

\date{}


\begin{abstract}
	We count the number of hyperbolic components of period $n$ that lie on the main molecule of the Mandelbrot set. We give a formula for how to compute the number of these hyperbolic components of period $n$ in terms of the divisors of $n$ and in the prime power case, an explicit formula is derived.
\end{abstract}

\maketitle


\section{\large Introduction}

\blfootnote{\emph{2020 Mathematics Subject Classification.} Primary 37F20; Secondary 37F10.}

Consider the map $f_c(z)=z^2+c$, where $c$ is the parameter variable and $z$ is the dynamic variable, both taking on complex values. Following \cite{M1}, we denote by $K(f_c)$ the filled Julia set, defined by taking the union of bounded orbits for $f_c$. Let the Mandelbrot set, $M$, be the compact subset of all parameter values $c$ such that $K(f_c)$ is connected. Recall that we may look at the period of the hyperbolic components of $M$ as shown in Figure 1. The main cardioid is the hyperbolic component that consists of the parameter values such that $f_c$ has an attracting fixed point in $\mathbb{C}$ \cite{Z}. 

\begin{definition}
	The \emph{main molecule} is the union of all hyperbolic components attached to the main cardioid through a chain of finitely many components. Let $M(n)$ denote the number of hyperbolic components of period $n$ on the main molecule.
\end{definition}

\begin{figure}[!h]
	\includegraphics[width=8.0cm, height=5.00cm]{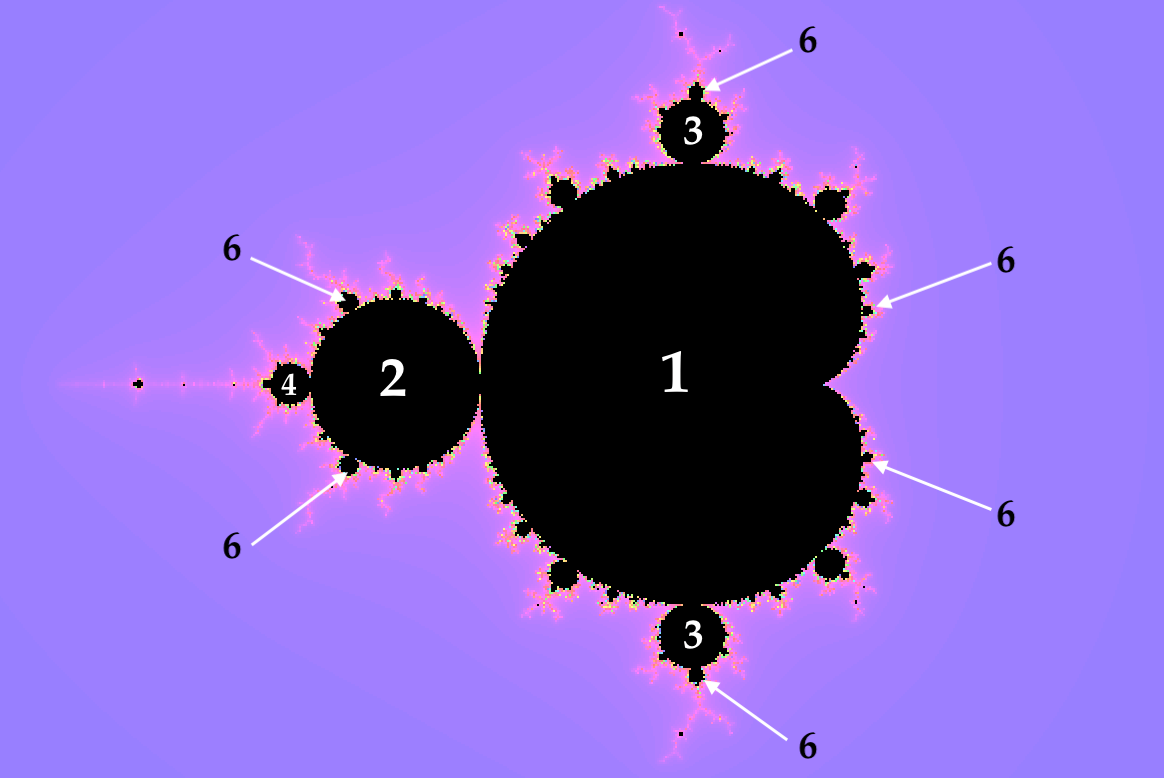}
	\centering
	\caption{Some periods of hyperbolic components of $M$. The hyperbolic components of period 6 on the main molecule are labeled in black.}
	\label{fig:slh}
\end{figure}

The closure of the main molecule is the locus of parameters for which the core entropy is zero. In this note, we are interested in computing $M(n)$ for various $n$ and we shall give formulas for $M(n)$. Apart from the main molecule, Lutzky has given formulas for the number of hyperbolic components of period $n$ \cite{L}, Kiwi and Rees have given formulas for certain hyperbolic components in the space of quadratic rational maps \cite{KR}, and Milnor and Poirier have studied hyperbolic components for polynomials of general degree \cite{MP}. In our case, the formulas for $M(n)$ turn out to be related to the Euler phi function and other combinatorial functions. For example, Figure 1 is an illustration that shows $M(6) = 6$. .


\section{\large Preliminary Results}

From \cite{BS}, there is a simple way to count how many hyperbolic components there are attached to specifically the main cardioid. 

\begin{lemma}
For any $n \in \mathbb{N}$, there are $\varphi(n)$ hyperbolic components of period $n$ attached to the main cardioid, where $\varphi$ is the Euler phi function.
\end{lemma}

\begin{proof}
	Let $H$ be a hyperbolic component of period $n$ attached to the main cardioid. If $n = 1$, then the main cardioid itself is the only such hyperbolic component so there is $\varphi(1) = 1$ component of period 1 attached to the main cardioid. Now suppose $n > 1$. Then, the root point of $H$ is a parabolic point of ray period $n$ \cite{M1}. For $c \neq 1/4$, let $f_c(z) = z^2 + c$ be a quadratic polynomial with a parabolic fixed point at $\alpha \in \mathbb{C}$ with multiplier, $\lambda$, a primitive $n$th root of unity. This multiplier must have the form $\lambda = e^{2\pi i \theta}$ where $\theta \in \mathbb{Q}/\mathbb{Z}$ is the rotation number. Each hyperbolic component has an associated rotation number. Because the number of primitive $n$th roots of unity is given by $\varphi(n)$, there are $\varphi(n)$ possible rotation numbers for $H$ hence $\varphi(n)$ hyperbolic components of period $n$. 
\end{proof}

\begin{corollary}
Let $p, n \in \mathbb{N}$. For any hyperbolic component $H$ of period $p$, there are $\varphi(n)$ hyperbolic components of period $pn$ attached to H.
\end{corollary}

\begin{proof}
	Douady and Hubbard \cite{DH} introduced a tuning map $i_H:M \rightarrow M$ whose image is a baby Mandelbrot set (that is, a homeomorphic copy of $M$) \cite{T}. From \cite{Z}, the idea of the tuning map is to take $c \in M$ and for each bounded component of the Fatou set corresponding to $f_c$, replace each of these bounded components with a homeomorphic copy of the filled Julia set corresponding to some parameter in $H$. This yields a set homeomorphic to the filled Julia set of $f_{i_H(c)}$. The tuning map $i_H$ also maps the main cardioid onto $H$ and maps hyperbolic components of period $n$ onto hyperbolic components of period $pn$. By Lemma 2.1, there must then be $\varphi(n)$ such hyperbolic components.
\end{proof}

\begin{theorem}
Let $n \in \mathbb{N}$ and write $n=d_1\cdot\cdot\cdot d_k$, where $d_1,...,d_k$ are divisors of $n$ with $d_i > 1$ for every $i \in \{1,...,k\}$. Then, we have

\[
M(n) =  \sum_{\substack{n=d_1...d_k \\ d_i>1}} \varphi(d_1)\cdot\cdot\cdot\varphi(d_k).
\]

\end{theorem}

\begin{example}
Before presenting the proof, let us compute $M(12)$, the number of hyperbolic components of period 12 attached to the main cardioid through a chain of finitely many components. We need to take into account all possible ordered partitions of 12 into its divisors. For instance, $6 \cdot 2$ and $2 \cdot 6$ represent two different such partitions. With this in mind, we have
\begin{align*}
M(12) & = \varphi(12) + \varphi(6)\varphi(2) + \varphi(2)\varphi(6) + \varphi(4)\varphi(3) + \varphi(3)\varphi(4) +  \varphi(3)\varphi(2)\varphi(2) + \\& \hspace{0.5cm}\varphi(2)\varphi(3)\varphi(2) + \varphi(2)\varphi(2)\varphi(3) \\
 & = 4 + (2 \cdot 1) + (1 \cdot 2) + (2 \cdot 2) + (2 \cdot 2) + (2 \cdot 1 \cdot 1) + (1 \cdot 2 \cdot 1) + (1 \cdot 1 \cdot 2) \\
 & = 22.
\end{align*}
Figure 2 illustrates the period 12 hyperbolic components attached to the main cardioid through the period 2 and 3 hyperbolic components. In particular, (A) corresponds to the terms $\varphi(2)\varphi(6)$, $\varphi(2)\varphi(3)\varphi(2)$, and $\varphi(2)\varphi(2)\varphi(3)$. (B) along with the other period 3 hyperbolic component corresponds to the terms $\varphi(3)\varphi(4)$ and $\varphi(3)\varphi(2)\varphi(2)$. 
\begin{figure}
	\centering
	\begin{subfigure}{.5\textwidth}
		\centering
		\includegraphics[width = 4.5cm, height = 5cm]{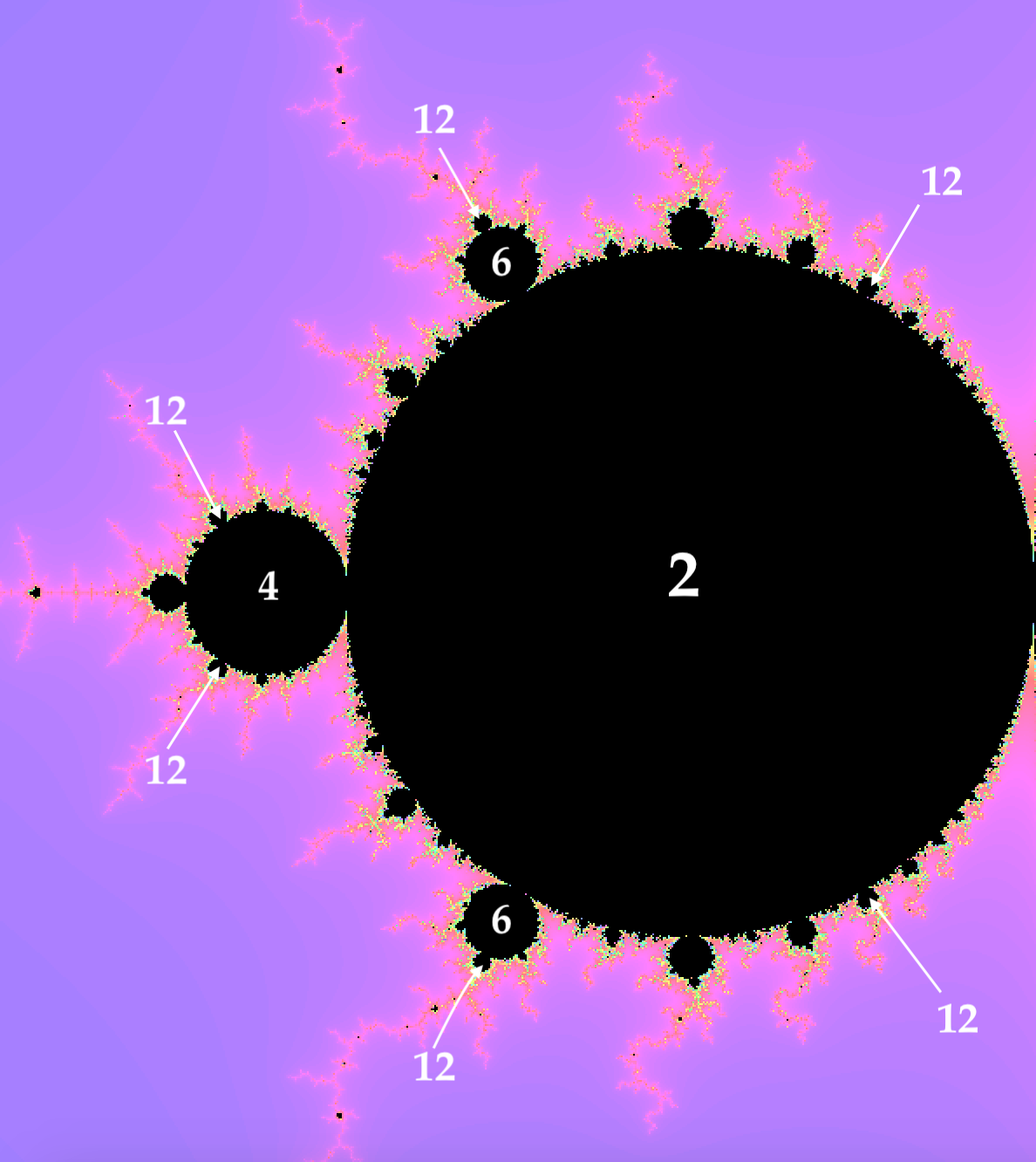}
		\caption{The period 2 hyperbolic component}
		\label{fig:sub1}
	\end{subfigure}%
	\begin{subfigure}{.5\textwidth}
		\centering
		\includegraphics[width = 6cm, height = 5cm ]{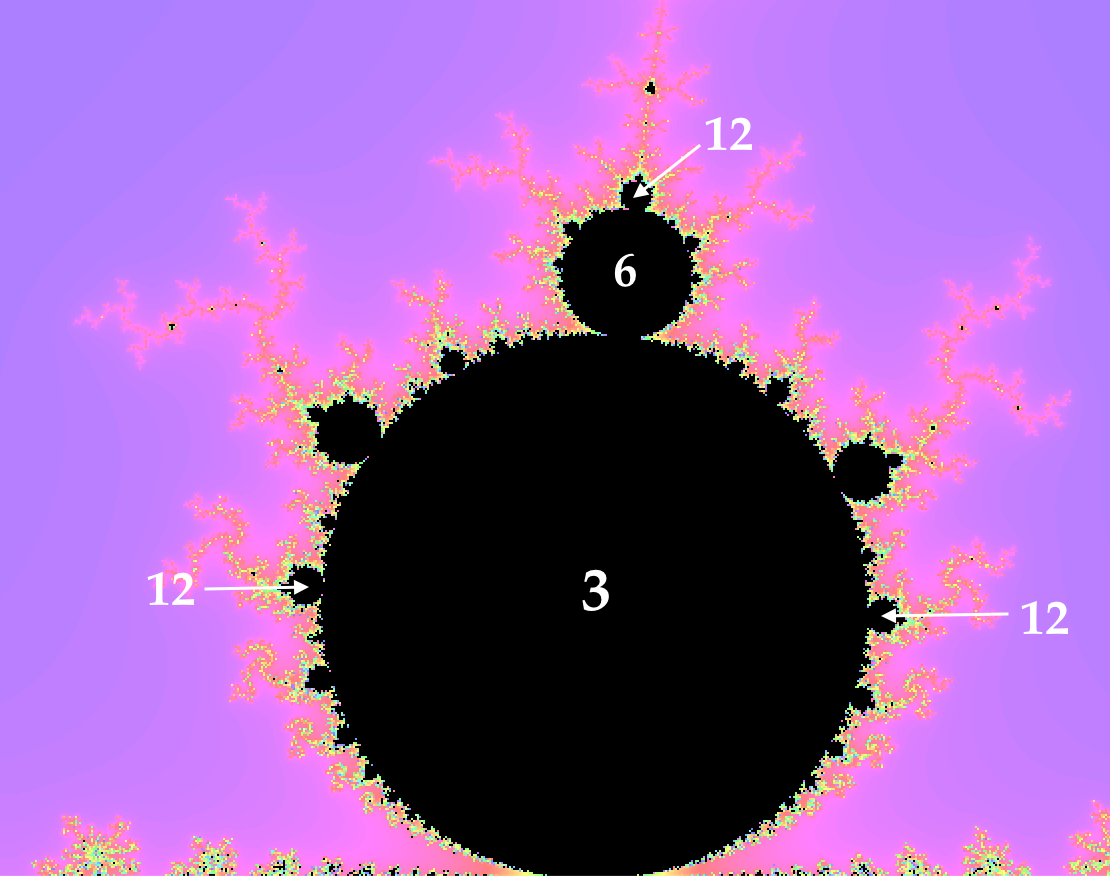}
		\caption{A period 3 hyperbolic component}
		\label{fig:sub2}
	\end{subfigure}
	\caption{Period 2 and 3 hyperbolic components attached to the main cardioid are shown along with the period 12 hyperbolic components attached to the main cardioid through a chain of finitely many components.}
	\label{fig:test}
\end{figure}

\end{example}

With the intuition from this example, we now present the proof of Theorem 2.3.

\begin{proof}
We proceed by induction on the number of divisors. For $n=1$, $M(1)=1$. Suppose $n \in \mathbb{N}$ has only one divisor, $d_1>1$. Then by Lemma 2.1, 
$$M(n) = \varphi(n) = \sum_{\substack{n=d_1 \\ d_1>1}} \varphi(d_1)$$
so the base case holds. Now suppose that the formula in the statement of the lemma holds for $k$ divisors, where $k \in \mathbb{N}$. We show that it also holds for $k+1$ divisors. Let $n \in \mathbb{N}$ have $k+1$ divisors. Then, using Lemma 2.1 we can count the number of hyperbolic components of period $n$ attached to the main cardioid and using Corollary 2.2 we can count the number of hyperbolic components of period $n$ attached to the components already attached to the main cardioid. Thus, we have
$$
M(n) = \varphi(n) + \varphi(d_1)M\Big(\frac{n}{d_1}\Big) + \dots + \varphi(d_{k+1})M\Big(\frac{n}{d_{k+1}}\Big).
$$
Applying the inductive hypothesis to each $M(\frac{n}{d_i})$ for $i \in \{1,...,k+1\}$, we may write the equation above as a sum over all the divisors of $n$. This is equivalent to:
$$
M(n) = \sum_{\substack{n=d_1...d_{k+1} \\ d_i>1}} \varphi(d_1)\dots \varphi(d_{k+1}).
$$
\end{proof}

With this formulation of $M(n)$, we can define $M(n)$ in a recursive way depending on the Euler $\varphi$ function and the divisors of $n$.

\begin{corollary}
For any $n \in \mathbb{N}$,
\[
 M(n) = \sum_{\substack{d | n \\ d>1}} \varphi(d)M\left(\frac{n}{d}\right).
\]
\end{corollary}

\begin{proof}
This follows from the inductive step of the proof of Theorem 2.3.
\end{proof}


\section{\large Prime Powers}

In two different ways, we will find an explicit formula for positive integers that are prime powers. First, we begin with an example for small prime powers.

\begin{example} Let $n, p \in \mathbb{N}$ with $p$ prime. 
\begin{enumerate}[(a)]
\item Suppose $n=p=p \cdot 1$. Then, $M(p)=\varphi(p)=p-1$ as all components of period $p$ lie on the main cardioid. \\
\item Now let $n=p^2=p^2 \cdot 1 = p \cdot p$. Then, there are $\varphi(p^2)$ components of period $p^2$ on the main cardioid. There are more components to count here because of the components of period $p$ attached to the main cardioid. On these components, there are components of period $p$ that must be counted. There are $\varphi(p) \cdot \varphi(p)$ such components. In total, 
$$
M(p^2) = \varphi(p^2) + (\varphi(p))^2=p(p-1)+(p-1)^2=(p-1)(2p-1).
$$
\item Finally, suppose $n=p^3=p^3 \cdot 1 = p^2 \cdot p = p \cdot p^2 = p \cdot p \cdot p$. Applying similar reasoning as before and computing, our total is now
$$
M(p^3) = \varphi(p^3) + \varphi(p^2)\varphi(p)+\varphi(p)\varphi(p^2)+(\varphi(p))^3 = (p-1)(2p-1)^2.
$$
\end{enumerate}
\end{example}

This example sheds light on the general case of when $n$ is an arbitrary prime power, leading to the following result.

\begin{theorem}
For any $n, k \in \mathbb{N}$, if $n = p^k$ , then $M(p^k)=(p-1)(2p-1)^{k-1}$.
\end{theorem}

\begin{proof}
Let $p \in \mathbb{N}$ be a prime.  Applying Corollary 2.4 for $n=p^k$, we obtain:
\begin{equation*}
\begin{split}
M(p^k) & = \sum_{1 \leq h < k} \varphi(p^h)M(p^{k-h}) \\
 & = \sum_{1\leq h < k} (p^{h-1}(p-1)^2(2p-1)^{k-h-1})+p^{k-1}(p-1) \\
 & = \frac{(2p-1)^{k-1}(p-1)^2}{p}\Big(\sum_{1\leq h < k}(\frac{p}{2p-1})^h \Big)+p^{k-1}(p-1) \\
 & = \frac{(2p-1)^{k-1}(p-1)^2}{p} \Bigg( \frac{1-(\frac{p}{2p-1})^k}{1-\frac{p}{2p-1}}-1\Bigg) +p^{k-1}(p-1) \\
 & = (p-1)(2p-1)^{k-1}.
\end{split}
\end{equation*}
\end{proof}


\section{\large Products of Distinct Primes}

Consider now the case where the positive integer $n$ is of the form $n=p_1...p_m$, where $p_1,..., p_m$ are distinct primes. Before considering the general case, let us again look at a simple example.

\begin{example}
Let $n=p_1p_2$, where $p_1$ and $p_2$ are distinct primes. As before, we need to consider the number of ways to write this product, namely $p_1p_2 \cdot 1$, $p_1 \cdot p_2$, and $p_2 \cdot p_1$. Recall to determine the number of hyperbolic components with period $p_1p_2$, we need to count the components of period $p_1p_2$ on the main cardioid, the components of period $p_2$ on the components of period $p_1$ on the main cardioid, and the components of period $p_1$ on the components of period $p_2$ on the main cardioid. Therefore, we must have
$$
M(p_1p_2)=\varphi(p_1p_2)+\varphi(p_1)\varphi(p_2)+\varphi(p_2)\varphi(p_1)=3\varphi(p_1 p_2)=3(p_1-1)(p_2-1).
$$
\end{example}

For the general case of $n=p_1...p_m$, we first need to determine how many ways we can write this product of primes in a similar manner as the example above. This can be done by a recursive process. Instead of thinking about how to write these products of primes, we can consider the equivalent problem of determining the number of ordered partitions of $\{1,...,m\}$. Let this number be represented by $N(m)$. Define $N(0)=1$. It is clear that $N(1) = 1$. The ordered partitions of $\{1,2\}$ are $(\{1\},\{2\})$, $(\{2\},\{1\})$, and $(\{1,2\},\{\})$ so $N(2) = 3$. One can check by hand that in fact $N(3)=13$ and $N(4)=75$. The numbers $N(m)$ are known as the ordered Bell numbers or Fubini numbers and from \cite{OAG}, they satisfy
$$N(m) \sim \frac{m!}{2(\log(2))^{m+1}}.$$
The following is a well-known lemma about the ordered Bell numbers.

\begin{lemma}
Let $n = p_1p_2 \cdots p_m$ be a product of distinct primes. Then, 
$$N(m)=\displaystyle\sum_{k=1}^{m}{{m}\choose{k}}N(m-k).$$
\end{lemma}

\begin{proof}
Let $1 \leq k \leq m$. Begin by choosing $k$ numbers from $\{1,...,m\}$. The number of ordered partitions of the remaining $m-k$ numbers is $N(m-k)$. Because there are ${{m}\choose{k}}$ ways of choosing $k$ numbers from $\{1,...,m\}$, there are ${{m}\choose{k}}N(m-k)$ ordered partitions fixing $k$ numbers and ordering the rest. As we are counting the total number of ordered partitions as $k$ ranges between 1 and $m$, we have
$$N(m)=\displaystyle\sum_{k=1}^{m}{{m}\choose{k}}N(m-k).$$
\end{proof}

In this case where $n$ is a product of distinct primes, the following theorem shows that the number of hyperbolic components of period $n$ on the main molecule is closely related to the ordered Bell numbers.

\begin{theorem}
Let $n = p_1p_2 \cdots p_m$, a product of distinct primes. Then, we have 
$$M(n)=\displaystyle{M(p_1 \dots p_m)=N(m)(p_1-1) \dots (p_m-1)} = N(m) \varphi(n).$$
\end{theorem}

\begin{proof}
By Theorem 2.3, if we write $n=d_1\cdot\cdot\cdot d_k$, where $d_1,...,d_k$ are divisors of $n$ with $d_i > 1$ for every $i \in \{1,...,k\}$, then we have
$$
M(p_1 \dots p_m) =  \sum_{\substack{n=d_1...d_k \\ d_i>1}} \varphi(d_1)\cdot\cdot\cdot\varphi(d_k).
$$
For $r_1 \in \{1, \dots, m \}$, we have that  $d_1 = p_{i_1} \dots p_{i_{r_1}}$ where $p_{i_1}, \dots ,p_{i_{r_1}}$ are $r_1$ of the primes $p_1, \dots, p_m$. Therefore,
$$\varphi(d_1) = \varphi(p_{i_1} \dots p_{i_{r_1}}) = \varphi(p_{i_1}) \dots \varphi(p_{i_{r_1}}).$$
For $r_2 \in \{1,...,m-r_1\}$, $d_2 = p_{j_1} \dots p_{j_{r_2}}$ where $p_{j_1}, \dots, p_{j_{r_2}}$ are $r_2$ of the remaining $m-r_1$ primes. Similarly,
$$\varphi(d_2) = \varphi(p_{j_1} \dots p_{j_{r_2}}) = \varphi(p_{j_1}) \dots \varphi(p_{j_{r_2}}).$$
Continuing in this way, at the $k$-th step, we will have exhausted all of the $m$ primes. We may then write $n = d_1 \dots d_k$ and note that
$$\varphi(d_1) \dots \varphi(d_k) = \varphi(p_1) \dots \varphi(p_m).$$
We are summing over all possible ways of writing $n$ in terms of its divisors. Each term in the sum is of the form $\varphi(p_1) \dots \varphi(p_m)$ and there are $N(m)$ possible ways to do so. Therefore, 
$$\displaystyle{M(p_1 \dots p_m)=N(m)\varphi(p_1) \dots \varphi(p_m) = N(m)(p_1-1) \dots (p_m-1)}=N(m)\varphi(n).$$
\end{proof}

Now let $p_m$ denote the $m$th prime number. As before, let $n=p_1 \cdots p_m$. We then have
\begin{align*}
	\frac{M(n)}{N(m)\cdot n} &= \frac{N(m)(p_1-1) \cdots (p_m-1)}{N(m) \cdot p_1 \cdots p_m} \\
	&= (1-\frac{1}{p_1}) \cdots (1-\frac{1}{p_m}) \\
	&< 1
\end{align*}
As $n \to \infty$ we must have $m \to \infty$ and it is well-known that the product above tends to zero as $m \to \infty$. Therefore, $M(n) = o(N(m) n)$ as $n \to \infty$.

\bigskip


 \section{\large Acknowledgements}

I would like to thank Giulio Tiozzo for his supervision and helpful comments.

 \bigskip


\bibliographystyle{amsplain}

\begin{thebibliography}{99}
	
\bibitem{BS} S. Bullett and P. Sentenac. \emph{Ordered orbits of the shift, square roots, and the devil's staircase}, Math. Proc. Camb. Phil. Soc., \textbf{115} (1994), 451-481.

\bibitem{DH} A. Douady and J. Hubbard. \emph{\'Etude Dynamique de Polyn\^{o}mes Complexes}, Publications Math. d'Orsay, Orsay, France: Universit\'e de Paris-Sud, D\'ep de Math. 1984.

\bibitem{KR} J. Kiwi and M. Rees. \emph{Counting hyperbolic components}, J. Lond. Math. Soc. (2), \textbf{88} (2013), 669-698. 

\bibitem{L} M. Lutzky. \emph{Counting hyperbolic components of the Mandelbrot set}, Phys. Lett. A, \textbf{177} (1993), 338-340.  

\bibitem{OAG} O. A. Gross, \emph{Preferential Arrangements}, Amer. Math. Monthly, \textbf{69} (1962), 4-8, DOI 10.2307/2312725.

\bibitem{M2} J. Milnor. \emph{Dynamics in One Complex Variable}, 3rd ed., Princeton Univ. Press, Oxfordshire, 2006. 

\bibitem {M1} J. Milnor. \emph{Periodic Orbits, External Rays and the Mandelbrot Set: An Expository Account}, G\'eom\'etrie complexe et syst\`emes dynamiques - Colloque en l'honneur d'Adrien Douady Orsay, \textbf{261} (2000), 57. 

\bibitem{MP} J. Milnor and A. Poirier. \emph{Hyperbolic components in spaces of polynomial maps}. arXiv:math/9202210, 1992.

\bibitem {T} G. Tiozzo. \emph{Topological Entropy of Quadratic Polynomials and Dimension of Sections of the Mandelbrot Set}, Adv. Math., \textbf{273} (2015), 651-715. 

\bibitem {Z} S. Zakeri. \emph{External Rays and the Real Slice of the Mandelbrot Set}, Ergodic Theory Dynam. Systems, \textbf{23} (2003), 637-660.



\end{thebibliography}

\end{document}